\documentclass[a4paper,10pt,reqno]{article}

\usepackage{amsmath}
\usepackage{amssymb}
\usepackage{amsopn}
\usepackage{amsthm}
\usepackage{amstext}
\newtheorem{proposition}{Proposition}[section]
\newtheorem{theorem}[proposition]{Theorem}
\newtheorem{lemma}[proposition]{Lemma}
\newtheorem{corollary}[proposition]{Corollary}
\newtheorem{definition}[proposition]{Definition}
\newtheorem{remark}[proposition]{Remark}

\def\R{\mathfrak R}
\def\R{\mathfrak R}

\begin{document}
\date{}
\title{The perturbation of the group inverse under the stable perturbation in a unital ring}
\author{Fapeng Du\thanks{E-mail: jsdfp@163.com}\\
   School of Mathematical \& Physical Sciences, Xuzhou Institute of Technology\\
  Xuzhou 221008, Jiangsu Province, P.R. China\\
 \and
 Yifeng Xue\thanks{Corresponding author, E-mail: yfxue@math.ecnu.edu.cn} \\
 Department of mathematics and Research Center for Operator Algebras\\
 East China Normal University, Shanghai 200241, P.R. China
}
\maketitle
\begin{abstract}
Let $\R $ be a ring with unit $1$ and $a\in \R ,\, \bar{a}=a+\delta a\in \R $ such that $a^\#$ exists. In this paper,
we mainly investigate the perturbation of the group inverse $a^\#$ on $\R$. Under the stable perturbation, we obtain
the explicit expressions of $\bar{a}^\#$. The results extend the main results in \cite{YFX, CX} and some related results
in \cite{YX}. As an application, we give the representation of the group inverse of the matrix
$\begin{bmatrix}d&b\\ c&0\end{bmatrix}$ on the ring $\R$ for certain $d,\, b,\,c\in\R$.
\end{abstract}

\section{Introduction}
Let $\R $ be a ring with unit and $a \in \R $. we consider an element $b\in \R $ and the following
equations:
\begin{center}
$(1)$ $aba=a$, \quad $(2)$ $bab=b$, \quad $(3)$ $a^kba=a^k$, \quad $(4)$ $ab=ba$.
\end{center}
If $b$ satisfies $(1)$, then $b$ is called a pseudo--inverse or $1$--inverse of $a$. In this case, $a$ is called regular.
The set of all $1$--inverse of $a$ is denoted by $a^{\{1\}}$; If $b$ satisfies $(2)$, then $b$ is called a $2$--inverse of
$a$, and $a$ is called anti--regular. The set of all $2$--inverse of $a$ is denoted by $a^{\{2\}}$; If $b$ satisfies
$(1)$ and $(2)$, then $b$ is called the generalized inverse of $a$, denoted by $a^+$; If $b$ satisfies $(2)$, $(3)$ and
$(4)$, then $b$ is called the Drazin inverse of $a$, denoted by $a^D$. The smallest integer $k$ is called the index of
$a$, denoted by $ind(a)$. If $ind(a)=1$, we say $a$ is group invertible and $b$ is the group inverse of $a$, denoted by
$a^\#$.

The notation so--called stable perturbation of an operator on Hilbert spaces and Banach spaces is introduced
by G. Chen and Y. Xue in \cite{CX1,CWX}.  Later the notation is generalized to Banach Algebra by Y. Xue in \cite{YFX} and
to Hilbert $C^*$--modules by Xu, Wei and Gu in \cite{XWG}. The stable perturbation of linear operator was widely
investigated by many authors. For examples, in \cite{CX2}, G. Chen and Y. Xue study the perturbation for Moore--Penrose
inverse of an operator on Hilbert spaces; Q. Xu, C. Song and Y. Wei studied the stable perturbation of the Drazin inverse
of the square matrices when $I-A^\pi-B^\pi$ is nonsingular in \cite{QCY} and Q. Huang and W. Zhai worked over the
perturbation of closed operators in \cite{QHL,QH}, etc.. Some further results can be found in \cite{DX1,DX2,DX3,DX4}.

The Drazin inverse has many applications in matrix theory, difference equations, differential equations and iterative
methods. In 1979, Campbell and Meyer proposed an open problem: how to find an explicit expression for the Drazin inverse
of the matrix $\begin{bmatrix}A&B\\C&D\end{bmatrix}$ in terms of its sub-block in \cite{CDM}? The representation of the
Drazin inverse of a triangular matrix $\begin{bmatrix}A&B\\0&D\end{bmatrix}$ has been given in \cite{GK,DPS,HWY}. In
\cite{DW}, Deng and Wei studied the Drazin inverse of the anti-triangular matrix $\begin{bmatrix}A&B\\C&0\end{bmatrix}$
and given its representation under some conditions.

In this paper, we investigate the stable perturbation for the group inverse of an element in a ring. Assume that
$1-a^\pi-\bar{a}^\pi$ is invertible, we present the expression of $\bar{a}^\#$ and $\bar{a}^D$.  This extends the related
results in \cite{YX, CX}. As an applications, we study the representation for the group inverse of the
anti--triangular matrix $\begin{bmatrix}d&b\\c&0\end{bmatrix}$ on the ring.

\section{Some Lemmas}
Throughout the paper, $\R $ is always a ring with the unit $1$. In this section, we give some lemmas:

\begin{lemma}\label{L1}
Let $a,\,b \in \R $. Then $1+ab$ is invertible if and only if $1+ba$ is invertible. In this case,
$(1+ab)^{-1}=1-a(1+ba)^{-1}b$ and
$$
(1+ab)^{-1}a=a(1+ba)^{-1},\ b(1+ab)^{-1}=(1+ba)^{-1}b.
$$
\end{lemma}

\begin{lemma}\label{LL1}
Let $a,\,b \in \R $. If $1+ab$ is left invertible, then so is $1+ba$.
\end{lemma}
\begin{proof} Let $c\in \R $ such that $c(1+ab)=1$. Then
$$
1+ba=1+bc(1+ab)a=1+bca(1+ba).
$$
Therefore, $(1-bca)(1+ba)=1$.
\end{proof}

\begin{lemma}\label{4L1}
Let $a$ be a nonzero element in $\R $ such that $a^+$ exists. If $s=a^+a+aa^+-1$ is invertible in $\R $, then $a^\#$
exists and $a^\#=a^+s^{-1}+(1-a^+a)s^{-1}a^+s^{-1}$.
\end{lemma}
\begin{proof} According to \cite{PC} or \cite[Theorem 4.5.9]{YX}, $a^\#$ exists. We now give the expression of $a^\#$ as
follows.

Put $p=a^+a$, $q=aa^+$. Then we have
\begin{equation}\label{4eqa}
ps=pq=sq,\ qs=qp=sp,\ sa=a^+a^2.
\end{equation}
Set $y=a^+s^{-1}$. Then by (\ref{4eqa}),
\begin{align*}
yp&=a^+s^{-1}a^+a=a^+aa^+s^{-1}=y=py,\\
pay&=a^+aaa^+s^{-1}=pqs^{-1}=p,\\
ypa&=a^+s^{-1}a^+aa=a^+a=p.
\end{align*}
Put $a_1=pap=pa$, $a_2=(1-p)ap=(1-p)a$. Then $a=a_1+a_2$ and it is easy to check that $a^\#=y+a_2y^2$. Using (\ref{4eqa}),
we can get that $a^\#=a^+s^{-1}+(1-a^+a)a(a^+s^{-1})^2=a^+s^{-1}+(1-a^+a)s^{-1}a^+s^{-1}.$
\end{proof}

Let $M_2(\R )$ denote the matrix ring of all $2\times 2$ matrices over $\R$ and let $1_2$ denote the unit of
$M_2(\R )$.
\begin{corollary}\label{c1}
Let $b,\,c\in \R $ have group inverse $b^\#$ and $c^\#$ respectively. Assume that
$k=b^\# b+c^\# c-1$ is invertible in $\R $. Then ${\begin{bmatrix}0&b\\ c&0\end{bmatrix}}^\#$ exists with
$
{\begin{bmatrix} 0&b\\ c&0\end{bmatrix}}^\#=\begin{bmatrix}0&k^{-1}c^\#k^{-1}\\ k^{-1}b^\#k^{-1}&0\end{bmatrix}.
$

In particular, when $b^\# bc^\# c=b^\# b$ and $c^\# cb^\# b=c^\# c$, ${\begin{bmatrix}0&b\\ c&0\end{bmatrix}}^\#=
\begin{bmatrix}0&b^\# bc^\#\\ c^\#cb^\#&0\end{bmatrix}$.
\end{corollary}
\begin{proof} Set $a=\begin{bmatrix}0&b\\ c&0\end{bmatrix}$. Then $a^+=\begin{bmatrix}0&c^\#\\ b^\#&0\end{bmatrix}$
and
$$
a^+a+aa^+-1_2=\begin{bmatrix}b^\# b+c^\# c-1&0\\ 0&b^\# b+c^\# c-1\end{bmatrix}=\begin{bmatrix}k\\ \ &k\end{bmatrix}
$$
is invertible in $M_2(\R )$. Noting that $bb^\#k^{-1}=k^{-1}cc^\#$. Thus, by Lemma \ref{4L1},
\begin{align*}
a^\#&=a^+\begin{bmatrix}k^{-1}\\ \ &k^{-1}\end{bmatrix}+(1_2-a^+a)\begin{bmatrix}k^{-1}\\ \ &k^{-1}\end{bmatrix}
a^+\begin{bmatrix}k^{-1}\\ \ &k^{-1}\end{bmatrix}\\
&=\begin{bmatrix}0&c^\# k^{-1}+(1-c^\# c)k^{-1}c^\# k^{-1}\\ b^\# k^{-1}+(1-b^\# b)k^{-1}b^\# k^{-1}&0\end{bmatrix}\\
&=\begin{bmatrix}0&k^{-1}c^\#k^{-1}\\ k^{-1}b^\#k^{-1}&0\end{bmatrix}.
\end{align*}

When $b^\# bc^\# c=b^\# b$ and $c^\# cb^\# b=c^\# c$, $k^{-1}=k$. In this case,
$
\begin{bmatrix}0&b\\ c&0\end{bmatrix}^\#=\begin{bmatrix}0&b^\#bc^\#\\ c^\#cb^\#&0\end{bmatrix}.
$
\end{proof}

\begin{lemma}\label{L2}
Let $a,\,b\in \R$ and $p$ be a non--trivial idempotent element in $\R$, i.e., $p\not=0,1$. Put $x=pap+pb(1-p)$.
\begin{enumerate}
\item[$(1)$] If $pap$ is group invertible and $(pap)(pap)^\#b(1-p)=pb(1-p)$, then $x$ is group invertible too and $x^\#=
(pap)^\#+[(pap)^\#]^2pb(1-p)$.
\item[$(2)$] If $x$ is group invertible, then so is the $pap$ and $(pap)(pap)^\#b(1-p)=pb(1-p)$.
\end{enumerate}
\end{lemma}
\begin{proof} (1) It is easy to check that $p(pap)^\#=(pap)^\# p=(pap)^\#$. Put $y=(pap)^\#+[(pap)^\#]^2pb(1-p)$. Then
$xyx=x$, $yxy=y$ and $xy=yx$, i.e., $y=x^\#$.

(2) Set $y_1=px^\# p$, $y_2=px^\#(1-p)$, $y_3=(1-p)x^\# p$ and $y_4=(1-p)x^\#(1-p)$. Then $x^\#=y_1+y_2+y_3+y_4$.
From $xx^\# x=x$, $x^\# xx^\#=x^\#$ and $xx^\#=x^\# x$, we can obtain that $y_3=y_4=0$ and
$$
(pap)y_1(pap)=pap,\quad y_1(pap)y_1=y_1,\quad y_1(pap)=(pap)y_1,\quad y_1(pap)pb(1-p)=pb(1-p),
$$
that is, $(pap)^\#=y_1$ and $(pap)(pap)^\#b(1-p)=pb(1-p)$.
\end{proof}

At the end of this section, we will introduce the notation of stable perturbation of an element in a ring.

Let $\mathcal A$ be a unital Banach algebra and $a\in\mathcal A$ such that $a^+$ exists. Let $\bar a=a+\delta a\in
\mathcal A$. Recall from \cite{YFX} that $\bar a$ is a stable perturbation of $a$ if $\bar a\mathcal A\cap(1-aa^+)\mathcal A=\{0\}$.
This notation can be easily extended to the case of ring as follows.

\begin{definition}
Let $a\in\R$ such that $a^+$ exists and let $\bar a=a+\delta a\in R$. We say $\bar a$ is a stable perturbation of $a$ if
$\bar a\R\cap(1-aa^+)\R=\{0\}$.
\end{definition}

Using the same methods as appeared in the proofs of \cite[Proposition 2.2]{YFX} and \cite[Theorem 2.4.7]{YX}, we can
obtain:
\begin{proposition}\label{P1}
Let $a\in \R $ and $\bar{a}=a+\delta a \in \R $ such that $a^+$ exists and $1+a^+\delta a$
is invertible in $\R $. Then the following statements are equivalent:
\begin{enumerate}
  \item [$(1)$] $\bar{a}^+$ exists and $\bar{a}^+=(1+a^+\delta a)^{-1}a^+$.
  \item [$(2)$] $\bar{a}\R  \cap (1-aa^+)\R =\{0\}$ $($ that is, $\bar a$ is a stable perturbation if $a)$.
  \item [$(3)$] $\bar{a}(1+a^+\delta a)^{-1}(1-a^+a)=0$.
  \item [$(4)$] $(1-aa^+)(1+\delta aa^+)^{-1}\bar{a}=0$.
  \item [$(5)$] $(1-aa^+)\delta a (1-a^+a)=(1-aa^+)\delta a (1+a^+\delta a)^{-1}a^+\delta a(1-a^+a)$.
  \item[$(6)$]  $\R \bar a\cap \R (1-a^+a)=\{0\}$.
\end{enumerate}
\end{proposition}

\section{Main results}
In this section, we investigate the stable perturbation for group inverse and Drazin inverse of an element $a$ in $\R$.

Let $a\in \R $ and $\bar{a}=a+\delta a \in \R $ such that $a^\#$ exists and
$1+a^\#\delta a$ is invertible in $\R $. Put $a^\pi=(1-a^\# a)$,
$\Phi(a)=1+\delta aa^\pi\delta a [(1+a^\#\delta a)^{-1}a^\#]^2$ and
$
B=\Phi(a)(1+\delta a a^\#),\ C(a)=a^\pi\delta a (1+a^\#\delta a)^{-1}a^\#.
$
These symbols will be used frequently in this section.
\begin{lemma}\label{P4}
Let $a\in \R $ and $\bar{a}=a+\delta a \in \R $ such that $a^\#$ exists and
$1+a^\#\delta a$ is invertible in $\R $. Suppose that $\Phi(a)$ is invertible, then $(Ba)^\#=Baa^\#B^{-1}a^\#B^{-1}.$
\end{lemma}
\begin{proof}
Put $P=aa^\#$.
 Noting that $\Phi(a)(1-P)=1-P$, we have $P\Phi(a) P=P\Phi(a)$, $\Phi^{-1}(a)(1-P)=(1-P)$ and
$PBP=PB$, $B^{-1}(1-P)=(1+\delta a a^\#)^{-1}(1-P)$, $a^\#B^{-1}(1-P)=0$, i.e., $a^\#B^{-1}=a^\#B^{-1}P$.
Thus, $BPB^{-1}Ba=Ba$ and
\begin{align*}
(Ba)(Baa^\#B^{-1}a^\#B^{-1})&=BPB^{-1}=(Baa^\#B^{-1}a^\#B^{-1})(Ba),\\
(Baa^\#B^{-1}a^\#B^{-1})(BPB^{-1})&=Baa^\#B^{-1}a^\#B^{-1}.
\end{align*}
These indicate $(Ba)^\#=Baa^\#B^{-1}a^\#B^{-1}.$
\end{proof}

\begin{theorem} \label{T1}
Let $a\in \R $ such that $a^\#$ exists. Let $\bar{a}=a+\delta a \in \R $ with $1+a^\#\delta a$
invertible in $\R $. Suppose that $\Phi(a)$ is invertible and $\bar{a}\R  \cap (1-aa^\#)\R =\{0\}$. Put
$D(a)=(1+a^\#\delta a )^{-1}a^\#\Phi^{-1}(a)$. Then $\bar{a}^\#$ exists with
$$\bar{a}^\#=(1+C(a))(D(a)+D^2(a)\delta a a^\pi)(1-C(a)).$$
\end{theorem}
{\it Proof.\ } Put $P=aa^\#$. By Proposition \ref{P1} (3), we have $a^\pi(1+\delta aa^\# )^{-1}\bar{a}=0$ and
$$
P\bar{a}(1+a^\# \delta a)^{-1}=a(aa^\#+a^\#\delta a)(1+a^\# \delta a)^{-1}
=a(1+a^\# \delta a-a^\pi) (1+a^\#\delta a )^{-1}=a.
$$
Thus, we have
\begin{align*}
&(1-C(a))\bar{a}(1+C(a))\\
&=[P+a^\pi(1+\delta a a^\#)^{-1}]\bar{a}[1+a^\pi \delta a(1+a^\# \delta a)^{-1}a^\#]\\
&=P\bar{a}[1+a^\pi \delta a(1+a^\# \delta a)^{-1}a^\#]\\
&=P\bar{a}+P\bar{a}a^\pi \delta a(1+a^\# \delta a)^{-1}a^\#\\
&=P\bar{a}+P\delta a a^\pi \delta a(1+a^\# \delta a)^{-1}a^\#\\
&=P\delta a+P[a+\delta a a^\pi \delta a(1+a^\# \delta a)^{-1}a^\#]\\
&=P\delta a(1-P)+P\delta a P+P[a+\delta a a^\pi \delta a(1+a^\# \delta a)^{-1}a^\#]\\
&=P\delta a(1-P)+P[\delta a +a+\delta a a^\pi \delta a(1+a^\# \delta a)^{-1}a^\#]P\\
&=P\delta a(1-P)+P[\bar{a}+\delta a a^\pi \delta a(1+a^\# \delta a)^{-1}a^\#]P\\
&=P\delta a(1-P)+P[\bar{a}(1+a^\# \delta a)^{-1}+\delta a a^\pi \delta a(1+a^\# \delta a)^{-1}a^\#(1+a^\# \delta a)^{-1}](1+a^\# \delta a)P\\
&=P\delta a(1-P)+P[a+\delta a a^\pi \delta a(1+a^\# \delta a)^{-1}a^\#(1+a^\# \delta a)^{-1}](1+a^\# \delta a)P\\
&=P\delta a(1-P)+P[a+\delta a a^\pi \delta a(1+a^\# \delta a)^{-1}a^\#(1+a^\# \delta a)^{-1}]a^\#(1+ \delta aa^\#)a\\
&=P\delta a(1-P)+P[aa^\#+\delta a a^\pi \delta a(1+a^\# \delta a)^{-1}a^\#(1+a^\# \delta a)^{-1}a^\#](1+ \delta aa^\#)a\\
&=P\delta a(1-P)+P[1+\delta a a^\pi \delta a((1+a^\# \delta a)^{-1}a^\#)^2](1+ \delta aa^\#)a\\
&=P\delta a(1-P)+P\Phi(a)(1+ \delta aa^\#)aP.
\end{align*}
By Lemma \ref{P4}, we have
$$
P(Ba)^\#P=PBaa^\#B^{-1}a^\#B^{-1}P=PBPB^{-1}a^\#B^{-1}P=a^\#B^{-1}=P(Ba)^\#
$$
and $PBa(Ba)^\# \delta a = P\delta a$. So $P(Ba)^\# P(Ba)P=P(Ba)^\#(Ba)$ and
\begin{align*}
P(Ba)PP(Ba)^\# P&=P(Ba)(Ba)^\#P=P(Ba)^\#(Ba)P=P(Ba)^\# P(Ba)P\\
P(Ba)^\# P(Ba)P(Ba)^\#P&=P(Ba)^\#(Ba)(Ba)^\#P=P(Ba)^\# P,\\
P(Ba)P(Ba)^\#P(Ba)P&=P(Ba)P(Ba)^\#(Ba)P=P(Ba)P,
\end{align*}
i.e., $(P(Ba)P)^\#=P(Ba)^\#=a^\#B^{-1}$. So $P(Ba)P(P(Ba)P)^\#=P$ and hence, we have by Lemma \ref{L2} (1),
$$[(1-C(a))\bar{a}(1+C(a))]^\#=a^\#B^{-1}+[a^\#B^{-1}]^2\delta a(1-P).$$
Therefore,
\begin{align*}
\bar{a}^\#&=(1+C(a))[(1-C(a))\bar{a}(1+C(a))]^\#(1-C(a))\\
&=(1+C(a))(D(a)+D^2(a)\delta a a^\pi)(1-C(a))\\
&=(1+a^\pi\delta a (1+a^\#\delta a)^{-1}a^\#)(1+a^\#\delta a)^{-1}a^\#\big[1+\delta aa^\pi\delta a[(1+a^\#\delta a)^{-1}a^\#]^2
\big]^{-1}\\
&\ \times\big[1+(1+a^\#\delta a)^{-1}a^\#\big[1+\delta aa^\pi\delta a[(1+a^\#\delta a)^{-1}a^\#]^2\big]^{-1}\delta aa^\pi\big]
(1-a^\pi\delta a (1+a^\#\delta a)^{-1}a^\#).\qed
\end{align*}

Now we consider the case when $a\in \R $ and $\bar{a}=a+\delta a \in \R $ such that $a^\#$, $\bar{a}^\#$ exist.
Firstly, we have
\begin{proposition}\label{P5}
Let $a\in \R $, $\bar{a}=a+\delta a \in \R $ such that $a^\#$, $\bar{a}^\#$ exist.
Then the following statements are equivalent:
\begin{enumerate}
  \item[$(1)$] $\R =\bar{a}\R \dotplus (1-aa^\#)\R =a\R \dotplus(1-\bar{a}\bar{a}^\#)\R =\R \bar{a}\dotplus \R (1-aa^\#)=\R a\dotplus \R (1-\bar{a}\bar{a}^\#)$.
  \item[$(2)$]  $K=K(a,\bar{a})=\bar{a}\bar{a}^\#+aa^\#-1$ is invertible.
  \item[$(3)$]  $\bar{a}\R \cap (1-aa^\#)\R =\{0\}$, $\R \bar{a}\cap \R (1-aa^\#)=\{0\}$ and $1+\delta a a^\#$ is invertible.
\end{enumerate}
\end{proposition}
\begin{proof}
$(1)\Rightarrow(2):$  Since $\R =\bar{a}\R \dotplus (1-aa^\#)\R =a\R \dotplus(1-\bar{a}\bar{a}^\#)\R $,
we have for any $y\in \R $, there are $y_1\in \R ,\,y_2\in \R $ such that
$$(1-\bar{a}\bar{a}^\#)y=(1-\bar{a}\bar{a}^\#)(1-aa^\#)y_1 ,\ \bar{a}\bar{a}^\#y=\bar{a}\bar{a}^\#aa^\#y_2.$$
Put $z=aa^\#y_2-(1-aa^\#)y_1$. Then
$$K(a,\bar{a})z=(\bar{a}\bar{a}^\#+aa^\#-1)(aa^\#y_2-(1-aa^\#)y_1)=y.$$

Since $\R =\R \bar{a}\dotplus \R (1-aa^\#)=\R a\dotplus \R (1-\bar{a}\bar{a}^\#)$,
we have for any $y \in \R $, there are $y_1,y_2\in \R$ such that
$$y(1-\bar{a}\bar{a}^\#)=y_1(1-aa^\#)(1-\bar{a}\bar{a}^\#), y\bar{a}\bar{a}^\#=y_2aa^\#\bar{a}\bar{a}^\#.$$
Put $z=y_2aa^\#-y_1(1-aa^\#)$. Then
$$zK(a,\bar{a})=(y_2aa^\#-y_1(1-aa^\#))(\bar{a}\bar{a}^\#+aa^\#-1)=y.$$
The above indicates $K(a,\bar{a})$ is invertible when we take $y=1$.

$(2)\Rightarrow(3):$
Let $y\in \bar{a}\R \cap (1-aa^\#)\R $. Then $\bar{a}\bar{a}^\#y=y,\,a^\# y=0$. Thus $K(a,\bar{a})y=0$ and hence
$y=0$, that is, $\bar{a}\R \cap (1-aa^\#)\R =\{0\}$. Similarly, we have $\R \bar{a}\cap \R (1-aa^\#)=\{0\}$.

Let $T=aK^{-1}\bar{a}^\#-a^\pi$. Since $\bar{a}a^\#aK^{-1}=\bar{a}$, we have  $(1+\delta a a^\#)T=K$,
that is, $(1+\delta a a^\#)$ has right inverse $TK^{-1}$.

Since $K^{-1}aa^\#\bar{a}=\bar{a}$, we have $(\bar{a}^\#K^{-1}a-a^\pi)(1+a^\# \delta a)=K$, that is, $1+a^\# \delta a$ has left
inverse $K^{-1}(\bar{a}^\#K^{-1}a-a^\pi)$. This indicates that $1+\delta a a^\#$ has left inverse
$1-\delta a K^{-1}(\bar{a}^\#K^{-1}a-a^\pi)a^\#$ by Lemma \ref{LL1}. Finally, $1+\delta a a^\#$ is invertible.

$(3)\Rightarrow (1):$ By Lemma \ref{L1}, $1+a^\#\delta a$ is also invertible. So from
$$
1+\delta aa^\#=\bar aa^\#+(1-aa^\#),\ 1+a^\#\delta a=a^\#\bar a+(1-aa^\#)
$$
and Lemma \ref{P1}, we get that
$$
\R =\bar a \R \dotplus(1-aa^\#)\R =\R \bar a\dotplus \R (1-aa^\#).
$$

We now prove that
$$
\R =a\R +(1-\bar{a}\bar{a}^\#)\R =\R a+\R (1-\bar{a}\bar{a}^\#),\ a\R \cap(1-\bar{a}\bar{a}^\#)\R =\R a\cap \R (1-\bar{a}\bar{a}^\#)=\{0\}.
$$

For any $y\in a\R \cap(1-\bar{a}\bar{a}^\#)\R $, we have $aa^\#y=y,\,\bar{a}y=0$. So
$(1+a^\#\delta a)y$ $=(1-a^\# a)y=0$ and hence $y=0$.

Similarly, we have $\R a\cap \R (1-\bar{a}\bar{a}^\#)=\{0\}$.

By Lemma \ref{P1}, $\bar a^+=(1+a^\#\delta a)^{-1}a^\#$ and $\bar a^+\bar a=(1+a^\#\delta a)^{-1}a^\#a(1+a^\#\delta a)$.
So $(1-\bar a^+\bar a)\R=(1+a^\#\delta a)^{-1}(1-a^\#a)\R$.
From $(1-\bar a^\#\bar a)(1-\bar a^+\bar a)=1-\bar a^+\bar a$, we get that $(1-\bar a^+\bar a)\R\subset
(1-\bar a^\#\bar a)\R$. Note that $a\R =a^\# \R $ and $(1+a^\#\delta a)a\R =a^\#(1+\delta aa^\#)\R =a^\# a\R $. So
$$
\R\supset a\R +(1-\bar a^\#\bar a)\R\supset (1+a^\#\delta a)^{-1}a^\# a\R +(1+a^\#\delta a)^{-1}(1-a^\# a)\R =\R .
$$

Similarly, we can get $\R a+\R (1-\bar a^\#\bar a)=\R $.
\end{proof}

Now we present a theorem which can be viewed as the inverse of Theorem \ref{T1} as follows:
\begin{theorem}\label{T2}
Let $a\in \R $ and $\bar{a}=a+\delta a \in \R $ such that $a^\#$, $\bar{a}^\#$ exist.
If $K(a,\bar{a})$ is invertible, then $\Phi(a)$ is invertible.
\end{theorem}
\begin{proof}
Since $K(a,\bar{a})$ is invertible, we have $\bar{a}\R \cap (1-aa^\#)\R =\{0\}$ and $1+\delta aa^\#$ is invertible in $\R $
by Propositon \ref{P5}.
Thus, from the proof of Theorem \ref{T1}, we have
\begin{align*}
(1-C(a))\bar{a}(1+C(a))&=P\delta a(1-P)+P\Phi(a)(1+ \delta aa^\#)aP\\
&=PBaP+P\delta a(1-P).
\end{align*}
Since $(1-C(a))\bar{a}(1+C(a))$ is group invertible, it follows from Lemma \ref{L2} (2) that $PBaP$ is group invertible
and $PBa(PBa)^\# \delta a (I-P)= P\delta a(I-P).$ Consequently,
\begin{align*}
[(1-C(a))\bar{a}(1+C(a))]^\#&=(1-C(a))\bar{a}^\#(1+C(a))\\
&=(PBa)^\#P+((PBa)^\#)^2\delta a(1-P).
\end{align*}
Thus,
\begin{align*}
(1-C(a))\bar{a}\bar{a}^\#(1+C(a))&=[PBaP+P\delta a(1-P)][(PBa)^\#P+((PBa)^\#)^2\delta a(1-P)]\\
&=PBa(PBa)^\#P+(PBa)^\#\delta a(1-P)\\
(1-C(a))K(a,\bar{a})(1+C(a))&=(1-C(a))\bar{a}\bar{a}^\#(1+C(a))-(1-C(a))a^\pi(1+C(a))\\
&=(1-C(a))\bar{a}\bar{a}^\#(1+C(a))-a^\pi(1+C(a))\\
&=PBa(PBa)^\#P+P(PBa)^\#\delta a(1-P)-(1-P)C(a)P-(1-P)
\end{align*}
Since $(1-C(a))K(a,\bar{a})(1+C(a))$ is invertible, we get that
$$
\rho(a)=(PBa)^\#PBa-(PBa)^\#\delta a C(a)=PBa[(PBa)^\#]^2(PBa-\delta a C(a))
$$
is invertible in $P\R P$. So we have $P=PBa[(Ba)^\#]^2(Ba-\delta a C(a))\rho^{-1}(a)$ and that $\Phi(a)$ has right inverse.

Set $E(a)=a^\#(1+\delta aa^\#)^{-1}\delta aa^\pi$. Then $1-E(a)=P+(1+a^\#\delta a)^{-1}a^\pi$ and $(1-E(a))^{-1}=
1+E(a)$. From Lemma \ref{P1}, we have
$\bar{a}(1+a^\#\delta a)^{-1}a^\pi=0$ and
\begin{align*}
a^\#(1+\delta aa^\#)^{-1}\bar{a}&=(1+a^\#\delta a)^{-1}a^\#\bar{a}=(1+a^\#\delta a)^{-1}(1+a^\#\delta a-a^\pi)\\
&=1-(1+a^\#\delta a)^{-1}a^\pi.
\end{align*}
Put $\psi(a)=1+[(1+a^\#\delta a)^{-1}a^\#]^2\delta aa^\pi\delta a$ and $R =(1+a^\#\delta a)\psi(a)$. Then
$(1-P)\psi(a)=1-P$, $P\psi(a)P=\psi(a)P$ and
\begin{align*}
&(1+E(a))\bar{a}(1-E(a))\\
&=[1+a^\#(1+\delta aa^\#)^{-1}\delta aa^\pi]\bar{a}[P+(1+a^\#\delta a)^{-1}a^\pi]\\
&=[1+a^\#(1+\delta aa^\#)^{-1}\delta aa^\pi]\bar{a}P\\
&=\bar{a}P+a^\#(1+\delta aa^\#)^{-1}\delta aa^\pi\bar{a}P\\
&=\bar{a}P+a^\#(1+\delta aa^\#)^{-1}\delta aa^\pi\delta aP\\
&=aP+\delta aP+a^\#(1+\delta aa^\#)^{-1}\delta aa^\pi\delta aP\\
&=(1-P)\delta aP+P\delta aP+aP+a^\#(1+\delta aa^\#)^{-1}\delta aa^\pi\delta aP\\
&=(1-P)\delta aP+P[\bar{a}+a^\#(1+\delta aa^\#)^{-1}\delta aa^\pi\delta a]P\\
&=(1-P)\delta aP+P(1+\delta aa^\#)[(1+\delta aa^\#)^{-1}\bar{a}+(1+\delta aa^\#)^{-1}a^\#(1+\delta aa^\#)^{-1}\delta aa^\pi\delta a]P\\
&=(1-P)\delta aP+Pa(1+a^\#\delta a)[a^\#(1+\delta aa^\#)^{-1}\bar{a}+a^\#(1+\delta aa^\#)^{-1}a^\#(1+\delta aa^\#)^{-1}\delta aa^\pi\delta a]P\\
&=(1-P)\delta aP+Pa(1+a^\#\delta a)[a^\#(1+\delta aa^\#)^{-1}\bar{a}+[(1+a^\#\delta a)^{-1}a^\#]^2\delta aa^\pi\delta a]P\\
&=(1-P)\delta aP+Pa(1+a^\#\delta a)[1+[(1+a^\#\delta a)^{-1}a^\#]^2\delta aa^\pi\delta a]P\\
&=(1-P)\delta aP+PaR P.
\end{align*}

Since $(1+E(a))\bar{a}(1-E(a))$ is group invertible, we can deduce that $aRP$ is group invertible and
\begin{align*}
[(1+E(a))\bar{a}(1-E(a))]^\#&=(1+E(a))\bar{a}^\#(1-E(a))\\
&=P(aRP)^\#+(1-P)\delta a((aRP)^\#)^2
\end{align*}
and
\begin{align*}
(1+E(a))\bar{a}\bar{a}^\#(1-E(a))&=[(1-P)\delta aP+PaR P][P(aRP)^\#+(1-P)\delta a((aRP)^\#)^2]\\
&=PaR (aRP)^\#+(1-P)\delta a(aRP)^\#.
\end{align*}
Thus, from the invertibility of $K(a,\bar{a})$, we get that
\begin{align*}
(1+E(a))&K(a,\bar{a})(1-E(a))\\
&=(1+E(a))\bar{a}\bar{a}^\#(1-E(a))-(1+E(a))a^\pi(1-E(a))\\
&=PaRP(aRP)^\#+(1-P)\delta aP(aRP)^\#-(1+E(a))a^\pi\\
&=PaRP(aRP)^\#+(1-P)\delta a(aRP)^\#-PE(a)(1-P)-(1-P)
\end{align*}
is invertible in $\R $ and hence
\begin{align*}
\eta(a)&=aRP (aRP)^\#-E(a)\delta a(aRP)^\#=[aRP -E(a)\delta a][(aRP)^\#]^2 aRP \\
&=[aRP -E(a)\delta a][(aRP)^\#]^2a(1+a^\#\delta a)\psi(a)P
\end{align*}
is invertible in $P\R P$. So $P\psi(a)P$ is left invertible and $\psi(a)$ is left invertible and hence $\Phi(a)$ is left
invertible by Lemma \ref{LL1}. Therefore, $\Phi(a)$ is invertible.
\end{proof}

Let $a\in \R $ such that $a^D$ exists and $ind(a)=s$. As we know if $a^D$ exists, then $a^l$ has group inverse $(a^l)^\#$
and $a^D=(a^l)^\# a^{l-1}$ for any $l\geq s$.

From Theorem \ref{T1} and Theorem \ref{T2}, we have the following corollary:
\begin{corollary}
Let $a$ and $b$ be nonzero elements in $R$ such that $a^D$ and $b^D$ exist. Put $s=ind(a)$ and $t=ind(b)$. Suppose that
$K(a,b)=bb^D+aa^D-1$ is invertible in $\R $. Then for any $l\geq s$ and $k\geq t$, we have
\begin{enumerate}
  \item[$(1)$]  $1+(a^D)^l(b^k-a^l)$ is invertible in $\R $ and $b^k \R \cap(1-a^D a)\R=\{0\}$.
  \item[$(2)$]  $W_{k,l}=1+E _{k,l}Z_{k,l}(1+(a^D)^l E _{k,l})^{-1}(a^D)^l$ is invertible in $\R $, here $E _{k,l}=b^k-a^l$ and
                 $Z_{k,l}=a^\pi E _{k,l}(a^D)^l(1+E _{k,l}(a^D)^l )^{-1}.$
  \item[$(3)$]  $b^D=(1+Z_{k,l})[H_{k,l}+H^2_{k,l}E _{k,l}a^\pi](1-Z_{k,l})b^{k-1}$,
  where $H_{k,l}=(1+(a^D)^l E _{k,l})^{-1}(a^D)^lW^{-1}_{k,l}.$
\end{enumerate}
\end{corollary}
\begin{proof}
Noting that $(a^D)^l=(a^l)^\#,\ aa^D=a^l(a^l)^\#,\ bb^D=b^k(b^k)^\#, \ l\geq s,\ k\geq t,$
we have
$$
K(a,b)=b^k(b^k)^\#+a^l(a^l)^\#-1,\quad 1+(a^D)^l(b^k-a^l)=1+(a^l)^\#(b^k-a^l).
$$
Applying Theorem \ref{T1} and Theorem \ref{T2} to $b^k$ and $a^l$, we can get the assertions.
\end{proof}

\section{The representation of the group inverse of certain matrix on $\R$}

As an application of Theorem \ref{T1} and Theorem \ref{T2}, we study the representation of the group inverse of $\begin{bmatrix}d&b\\ c&0\end{bmatrix}$
on the ring $\R$.

\begin{proposition}\label{4P1}
Let $b,\, c,\, d\in \R $. Suppose that $b^\#$ and $c^\#$ exist and $k=b^\# b+c^\# c-1$ is
invertible. If $b^\pi d=0$ or $dc^\pi=0$, then $\begin{bmatrix}d&b\\ c&0\end{bmatrix}^\#$ exists and
$$
\begin{bmatrix}d&b\\ c&0\end{bmatrix}^\#=\begin{bmatrix}-k^{-1}c^\#k^{-1}b^\#k^{-1}dc^\pi k^{-1}&k^{-1}c^\#k^{-1}\\
k^{-1}b^\#k^{-1}(1+dk^{-1}c^\#k^{-1}b^\#k^{-1}dc^\pi k^{-1})&-k^{-1}b^\#k^{-1}dk^{-1}c^\#k^{-1}\end{bmatrix}
$$
if $b^\pi d=0$. When $dc^\pi=0$, we have
$$
\begin{bmatrix}d&b\\ c&0\end{bmatrix}^\#=
\begin{bmatrix}-k^{-1}b^\pi dk^{-1}c^\#k^{-1}b^\#k^{-1}&(1+k^{-1}b^\pi dk^{-1}c^\#k^{-1}b^\#d)k^{-1}c^\#k^{-1}\\
k^{-1}b^\#k^{-1}&-k^{-1}b^\#k^{-1}dk^{-1}c^\#k^{-1}\end{bmatrix}.
$$
\end{proposition}
{\it Proof.\ } Set $a=\begin{bmatrix}0&b\\ c&0\end{bmatrix}$, $\delta a=\begin{bmatrix}d&0\\ 0&0\end{bmatrix}$
and $\bar a=\begin{bmatrix}d&b\\ c&0\end{bmatrix}$. Since $b^\#bk=kc^\#c$, $c^\#ck=kb^\#b$, it follows from
Corollary \ref{c1} that
\begin{align*}
1_2+a^\#\delta a&=1_2+\begin{bmatrix}0&b^\#bk^{-1}c^\#k^{-1}\\ c^\#ck^{-1}b^\#k^{-1}&0\end{bmatrix}
\begin{bmatrix}d&0\\ 0&0\end{bmatrix}=\begin{bmatrix}1&0\\ k^{-1}b^\#k^{-1}d&1\end{bmatrix}\\
(1_2+a^\#\delta a)^{-1}a^\#&=\begin{bmatrix}1&0\\ -k^{-1}b^\#k^{-1}d&1\end{bmatrix}
                           \begin{bmatrix}0&k^{-1}c^\#k^{-1}\\ k^{-1}b^\#k^{-1}&0\end{bmatrix}\\
&=\begin{bmatrix}0&k^{-1}c^\#k^{-1}\\ k^{-1}b^\#k^{-1}&-k^{-1}b^\#k^{-1}dk^{-1}c^\#k^{-1}\end{bmatrix}\\
aa^\#&=\begin{bmatrix}0&b\\ c&0\end{bmatrix}\begin{bmatrix}0&b^\#bk^{-1}c^\#k^{-1}\\ c^\#ck^{-1}b^\#k^{-1}&0\end{bmatrix}\\
&=\begin{bmatrix}bc^\#ck^{-1}b^\#k^{-1}&0\\ 0&cb^\#bk^{-1}c^\#k^{-1}\end{bmatrix}\\
&=\begin{bmatrix}bb^\#k^{-1}&0\\ 0&cc^\#k^{-1}\end{bmatrix}\\
a^\pi&=1-aa^\#=\begin{bmatrix}-c^\pi k^{-1}&0\\ 0&-b^\pi k^{-1}\end{bmatrix}\\
\bar{a}(1_2+a^\#\delta a)^{-1}a^\pi&=\begin{bmatrix}d&b\\ c&0\end{bmatrix}\begin{bmatrix}1&0\\ -c^\#ck^{-1}b^\#k^{-1}d&1\end{bmatrix}
                           \begin{bmatrix}-c^\pi k^{-1}&0\\ 0&-b^\pi k^{-1}\end{bmatrix}\\
&=\begin{bmatrix}d-bc^\#ck^{-1}b^\#k^{-1}d&b\\ c&0\end{bmatrix}\begin{bmatrix}-c^\pi k^{-1}&0\\ 0&-b^\pi k^{-1}\end{bmatrix}\\
&=\begin{bmatrix}-c^\pi k^{-1}d&b\\ c&0\end{bmatrix}\begin{bmatrix}-c^\pi k^{-1}&0\\ 0&-b^\pi k^{-1}\end{bmatrix}
=\begin{bmatrix}k^{-1}b^\pi dc^\pi k^{-1}&0\\ 0&0\end{bmatrix}
\end{align*}
and
\begin{align*}
\delta a a^\pi \delta a&=\begin{bmatrix}d&0\\ 0&0\end{bmatrix}\begin{bmatrix}-c^\pi k^{-1}&0\\ 0&-b^\pi k^{-1}\end{bmatrix}
                                 \begin{bmatrix}d&0\\ 0&0\end{bmatrix}=\begin{bmatrix}-dc^\pi k^{-1}d&0\\ 0&0\end{bmatrix}\\
a^\pi \delta a&=\begin{bmatrix}-k^{-1}b^\pi d&0\\ 0&0\end{bmatrix},\quad
\delta aa^\pi=\begin{bmatrix}-dc^\pi k^{-1}&0\\ 0&0\end{bmatrix}.
\end{align*}
If $b^\pi d=0$ or $d c^\pi=0$, then $\bar{a}(1+a^\#\delta a)^{-1}a^\pi=0$ and $\delta a a^\pi \delta a=0$. Thus,
$\Phi(a)=1_2$ and $D(a)=(1+a^\#\delta a)^{-1}a^\# \Phi^{-1}(a)=(1+a^\#\delta a)^{-1}a^\#$.

When $b^\pi d=0$, $C(a)=a^\pi \delta a(1+a^\#\delta a)^{-1}a^\#=0$ and
\begin{align*}
D(a)\delta a a^\pi&=\begin{bmatrix}0&k^{-1}c^\#k^{-1}\\ k^{-1}b^\#k^{-1}&-k^{-1}b^\#k^{-1}dk^{-1}c^\#k^{-1}\end{bmatrix}
\begin{bmatrix}-dc^\pi k^{-1}&0\\ 0&0\end{bmatrix}\\
&=\begin{bmatrix}0&0\\-k^{-1}b^\#k^{-1}dc^\pi k^{-1}&0\end{bmatrix}.
\end{align*}
By Theorem \ref{T1}, we have
\begin{align*}
\bar{a}^\#&=(1_2+C(a))(D(a)+D^2(a)\delta a a^\pi)(1_2-C(a))\\
&=D(a)(1_2+D(a)\delta a a^\pi)\\
&=\begin{bmatrix}0&k^{-1}c^\#k^{-1}\\ k^{-1}b^\#k^{-1}&-k^{-1}b^\#k^{-1}dk^{-1}c^\#k^{-1}\end{bmatrix}
\begin{bmatrix}1&0\\-k^{-1}b^\#k^{-1}dc^\pi k^{-1}&1\end{bmatrix}\\
&=\begin{bmatrix}a_1&k^{-1}c^\#k^{-1}\\ a_2&-k^{-1}b^\#k^{-1}dk^{-1}c^\#k^{-1}\end{bmatrix},
\end{align*}
where $a_1=-k^{-1}c^\#k^{-2}b^\# k^{-1}dc^\pi k^{-1}$,
$
a_2=k^{-1}b^\#k^{-1}+k^{-1}b^\#k^{-1}dk^{-1}c^\#k^{-2}b^\#k^{-1}dc^\pi k^{-1}.
$
Since $cc^\#b^\#=kb^\#$, it follows that
$$
c^\#k^{-2}b^\#=c^\#(c^\#c)k^{-1}k^{-1}b^\#=c^\# k^{-1}k^{-1}c^\#cb^\#=c^\#k^{-1}b^\#.
$$
So $a_1=-k^{-1}c^\#k^{-1}b^\# k^{-1}dc^\pi k^{-1}$, $a_2=k^{-1}b^\#k^{-1}(1+dk^{-1}c^\#k^{-1}b^\#k^{-1}dc^\pi k^{-1})$.

When $dc^\pi=0$, we have by Theorem \ref{T1},
\begin{align*}
\bar a^\#&=(1_2+C(a))D(a)(1_2-C(a))=(1_2+C(a))D(a)\\
&=\begin{bmatrix}1&-k^{-1}b^\pi dk^{-1}c^\#k^{-1}\\ 0&1\end{bmatrix}
\begin{bmatrix}0&k^{-1}c^\#k^{-1}\\ k^{-1}b^\#k^{-1}&-k^{-1}b^\#k^{-1}dk^{-1}c^\#k^{-1}\end{bmatrix}\\
&=\begin{bmatrix}-k^{-1}b^\pi dk^{-1}c^\#k^{-1}b^\#k^{-1}&(1+k^{-1}b^\pi dk^{-1}c^\#k^{-1}b^\#k^{-1}d)k^{-1}c^\#k^{-1}\\
k^{-1}b^\#k^{-1}&-k^{-1}b^\#k^{-1}dk^{-1}c^\#k^{-1}\end{bmatrix}.\qed
\end{align*}

Combining Proposition \ref{4P1} with Corollary \ref{c1}, we have
\begin{corollary}\label{c2}
Let $b,\, c,\, d\in \R $. Assume that $b^\#$ and $c^\#$ exist and satisfy conditions: $b^\# bc^\# c=b^\# b$,
$c^\# cb^\# b=c^\# c$.
\begin{enumerate}
\item[$(1)$] If $b^\pi d=0$, then
$
\begin{bmatrix}d&b\\ c&0\end{bmatrix}^\#=\begin{bmatrix}b^\#bc^\#b^\#db^\pi&b^\#bc^\#
\\ c^\#cb^\#(1-db^\#bc^\#b^\#db^\pi) &- c^\#cb^\#db^\#bc^\#\end{bmatrix}.
$
\item[$(2)$] If $dc^\pi=0$, then
$
\begin{bmatrix}d&b\\ c&0\end{bmatrix}^\#=\begin{bmatrix}b^\pi db^\#bc^\#b^\#&(1-b^\pi db^\#bc^\#b^\#d)b^\#bc^\#\\
c^\#cb^\# &-c^\#cb^\#db^\#bc^\#\end{bmatrix}.
$
\end{enumerate}
\end{corollary}

Recall from \cite{RH} that an involution * on $\R$ is an involutory anti--automorphism, that is,
$$
(a^*)^*=a,\ (a+b)^*=a^*+b^*,\  (ab)^*=b^*a^*,\ a^*=0\ \text{if and only if}\ a=0
$$
and $\R$ is called the $*$--ring if $\R$ has an involution.

\begin{corollary}\label{c3}
Let $\R$ be a $*$--ring with unit $1$ and let $p$ be a nonzero idempotent element in $\R$. Then
$$
\begin{bmatrix}pp^*&p\\ p&0\end{bmatrix}^\#=\begin{bmatrix}pp^*(1-p)&p\\ p-(pp^*)^2(1-p)&-pp^*p\end{bmatrix},\quad
\begin{bmatrix}p^*p&p\\ p&0\end{bmatrix}^\#=\begin{bmatrix}(1-p)p^*p&p-(1-p)(p^*p)^2\\ p&-pp^*p\end{bmatrix}.
$$
\end{corollary}
\begin{proof}
Since $p^\#=p$, we can get the assertions easily by using Corollary \ref{c2}.
\end{proof}

\begin{remark}
{\rm (1) If $\R$ is a skew field and $b=c$ in Proposition \ref{4P1}, the conclusion of Proposition \ref{4P1} is contained
in \cite{ZB}.

(2) Let $p$ be an idempotent matrix. The group inverse of $\begin{bmatrix}a&b\\ c&0\end{bmatrix}$ is given in \cite{CT}
for some $a,\,b,\,c\in\{pp^*,\,p,\,p^*\}$. The group inverse of this type of matrices is also discussed in \cite{DC}.
}
\end{remark}

\vskip0.2cm
\noindent{\bf{Acknowledgement.}} The authors thank to the referee for his (or her) helpful comments and suggestions.

\end{document}